\documentclass{amsart}

\usepackage{amssymb}
\usepackage{amsmath}
\usepackage{amsthm}
\usepackage{amsbsy}
\usepackage{graphicx}
\usepackage{bm}

\theoremstyle{plain}
\newtheorem{theorem}{Theorem}[section]
\newtheorem{lemma}[theorem]{Lemma}
\newtheorem{corollary}[theorem]{Corollary}

\newtheorem*{lem}{Lemma}

\theoremstyle{definition}
\newtheorem{definition}[theorem]{Definition}

\theoremstyle{remark}

\newcommand{\R}{\mathbb{R}}

\newcommand{\m}{\mathcal{M}}
\renewcommand{\S}{\mathcal{S}}

\newcommand{\del}{\partial}

\newcommand{\ep}{\epsilon}

\newcommand{\si}{\Sigma}
\newcommand{\rt}{\rightarrow}

\newcommand{\la}{\lambda}
\newcommand{\de}{\delta}
\newcommand{\D}{{\mathbb B}^2}

\newcommand{\incl}{\hookrightarrow}
\newcommand{\be}{\beta}
\begin{document}

\title[Bounded mean curvature surfaces]{Surfaces of bounded mean curvature\\ in Riemannian
manifolds}

\author{Siddhartha Gadgil}

\address{   Department of Mathematics\\
        Indian Institute of Science\\
        Bangalore 560003, India}

\email{gadgil@math.iisc.ernet.in}

\author{Harish Seshadri}

\address{   Department of Mathematics\\
        Indian Institute of Science\\
        Bangalore 560003, India}

\email{harish@math.iisc.ernet.in}

\date{\today}


\thanks{UGC support. The second author was supported by DST Grant No.
SR/S4/MS-283/05}
\begin{abstract}

Consider a sequence of closed, orientable surfaces of fixed genus
$g$ in a Riemannian manifold $M$ with uniform upper bounds on mean
curvature and area. We show that on passing to a subsequence and
choosing appropriate parametrisations, the inclusion maps converge
in $C^0$ to a map from a surface of genus $g$ to $M$.

We also show that, on passing to a further subsequence, the distance functions corresponding to pullback metrics converge to a pseudo-metric of fractal dimension two.

As a corollary, we obtain a purely geometric result. Namely, we show that bounds on the mean curvature, area and genus of a surface $F\subset M$ together with bounds on the geometry of $M$ give an upper bound on the diameter of $F$.

Our proof is modelled on Gromov's compactness theorem for
$J$-holomorphic curves.

\end{abstract}

\maketitle

\setcounter{tocdepth}{1}

\section{Introduction}

The study of families of minimal surfaces and, more generally,
constant mean curvature surfaces in Riemannian manifolds is a
classical topic in differential geometry. Minimal surfaces in
3-manifolds, in particular, has received a lot of attention. In
this paper, we focus our attention on closed (compact without
boundary) surfaces in closed Riemannian $n$-manifolds with \emph{bounded mean curvature}, generalising the case of surfaces of minimal surfaces (which are characterised by having mean curvature zero).

We prove a compactness result in a general setting. Let $(M,g)$ be
a closed, Riemannian $n$-manifold and let $H_0\geq 0$ and $A_0>0$
be fixed constants. Consider a sequence of closed, connected
orientable surfaces $F_j$ in $M$ of a fixed genus $m$ with
\begin{enumerate}
\item the norm of the mean curvature vector field bounded above by
$H_0$ and
\item area bounded above by $A_0$.
\end{enumerate}

Let $i_j:F_j\to M$ denote the inclusion maps. Let $F$ be a
smooth surface of genus $m$. Our main result says that after
choosing appropriate parametrisations a subsequence of the
surfaces converges in a $C^0$-sense to a limiting (not necessarily
embedded) surface.

\begin{theorem}\label{mai}
There are homeomorphisms $\varphi_j:F_j\to F$ such that, after
passing to a subsequence, the maps $i_j\circ \varphi_j^{-1}$
converge in the $C^0$ topology to a map $i_j:F\to M$.
\end{theorem}

On passing to a further subsequence, we show that the distance functions corresponding to the pullback metrics
converge to a pseudo-metric. We show further that the limit is in an appropriate sense $2$-dimensional.

\begin{theorem}\label{dim}
On passing to a subsequence, the distance functions $d_j$ on $F$
corresponding to the pullback metrics $g_j=(i_j\circ
\varphi_j^{-1})^*g$ converge uniformly to a (continuous)
pseudo-metric $d$ on $F$. Moreover $d$ has fractal dimension $2$.
\end{theorem}

As an application of Theorems~\ref{mai} and~\ref{dim}, we prove the following purely geometric result.

\begin{theorem}\label{diam}
Given a Riemannian manifold $(M,g)$ and constants $A_0>0$ and
$H_0\geq 0$ and an integer $m>0$, there is a constant $D=D(M,g,
A_0, H_0,m)$, depending only on $A_0$, $H_0$, $m$ and $(M,g)$, so
that any surface $F\subset M$ of genus $m$ with the norm of the
mean curvature vector bounded above by $H_0$ and area at most
$A_0$ has diameter at most $D$.
\end{theorem}
\begin{proof}
We proceed by contradiction. Suppose $F_j\subset M$ is a sequence of surfaces satisfying the hypothesis with
$diam(F_j)\to \infty$. By Theorem~\ref{mai}, on passing to a subsequence (which we also denote $F_j$) we can
construct a family of diffeomorphisms $\varphi_j:F_j\to F$ to a fixed surface so that the associated maps $F\to M$
converge. Let $d_j:F\times F\to\R$ be the corresponding distance functions on $F$. Then
$diam(F_j)=sup_{(p,q)\in F\times F} d_j(p,q)\to\infty$.

By the first statement of Theorem~\ref{dim}, on passing to a further subsequence, the functions $d_j:F\times F\to \R$ converge to a continuous function $d:F\times F\to\R$, which is bounded as $F\times F$ is compact. It follows that the functions $d_j$ are uniformly bounded above by the maximum of $d$, a contradiction.
\end{proof}




There is a large body of literature dealing with families of
minimal surfaces in Riemannian 3-manifolds. The foundational work
of W. Allard ~\cite{al} deals with weak convergence of minimal
surfaces in $n$-manifolds. For surfaces in 3-manifolds, M. T.
Anderson proved the following ~\cite{an}: Let $\m _n$ denote the
space of minimal embeddings of a closed surface of genus $\ge n$
in a complete 3-manifold, endowed with the weak topology as a
subset of the space of 2-varifolds. Then the boundary $\partial \m
_n$ is contained in $\m_{\frac {n}{2}}$. As a corollary, it is
shown that if the 3-manifold is compact and has negative sectional
curvature, then $\m_2$ is compact.

For 3-manifolds of positive curvature, H. Choi and R. Schoen prove
~\cite{cs} the following compactness result: Let $N$ be a closed
3-manifold of positive Ricci curvature. Then the space of closed
embedded minimal surfaces of fixed topological type, endowed with
the $C^k$ topology, is compact if $k \ge 2$. In ~\cite{wh} B.
White generalized the result of Choi and Schoen to stationary
points of arbitrary elliptic functionals defined on the space of
embeddings of a compact surface in a 3-manifold, minimal surfaces
being stationary points of the area functional. White's result is
that compactness holds for such surfaces if we assume a bound on
the area.

More recently, W. Minicozzi and T. Colding have studied ~\cite{cm}
sequences of minimal surfaces in 3-manifolds of bounded genus
without bounds on area.

\subsection*{Outline of the proof of Theorem~\ref{mai}}
Our compactness result and proof are modelled on Gromov's
compactness theorem for $J$-holomorphic curves. However the ingredients in our situation -
for instance the Schwarz lemma and the monotonicity lemma, need to be proved using different
techniques than those for $J$-holomorphic curves. Furthermore, unlike the case of
$J$-holomorphic curves (or minimal surfaces), the surfaces we
consider do not satisfy an elliptic partial differential equation,
and hence we do not have regularity results.

Consider henceforth a sequence of surfaces $ F_{i} $ in the
manifold $M$ satisfying the above bounds on the mean curvature,
area and genus. The surfaces $F_i$ have Riemannian metrics
obtained by restriction from $M$. We shall consider distances
with respect to this metric. We shall frequently replace the
given sequence by a subsequence, for which we continue to use the
same notation.

The first elementary observation (see Section~\ref{S:prel}) is
that the upper bound on mean curvature ensures a uniform upper
bound on the sectional curvature of the surfaces.

The basic strategy of the proof is to puncture the surfaces about
an $\epsilon$-net and on the complement, to consider the complete
hyperbolic metric in the conformal class of $\iota^\ast g$. Here
$\epsilon>0$ is a sufficiently small constant depending only on
the geometry of the ambient manifold $M$. A monotonicity lemma we
prove in Section~\ref{S:monot} shows that there is a uniform bound
on the size of the $\epsilon$-net. Hence by passing to a
subsequence we may assume that the topological type of the
punctured surfaces is fixed.

The Bers-Mumford compactness theorem says that, on passing to a
subsequence, the hyperbolic surfaces have a limit which is a complete hyperbolic surface. The limiting surfaces in general has additional cusps created by pinching curves. Our goal is to construct a corresponding limit of
maps. This depends on relating the hyperbolic metric on the surfaces
to the metric obtained from $M$.

Specifically, we show that the away from the cusps the identity
map from the surface with the hyperbolic metric to itself with the
metric restricted from $M$ is uniformly continuous. The first step
(Section~\ref{S:hyp}) is an argument that says that we have a form
of uniform continuity at one scale. This uses an extremal length
argument and the fact that an $\epsilon$-net has been deleted. We
then need an appropriate Schwarz lemma to conclude uniform
continuity at all stages.

The version of Schwarz lemma we prove (in Section~\ref{S:schwarz}) and
use is for discs with small diameter with an upper bound on the
sectional curvature given. However this cannot be applied directly as
it requires a lower bound on the injectivity radius at the origin (of the pullback metric). We
apply the Schwarz lemma indirectly by constructing a lift of an
appropriate disc under the exponential map. Such a lift is obtained
(in Section~\ref{S:lift}) by a geometric argument, making crucial use
of an upper bound on the perimeter of the disc.

Following Gromov's proof of the compactness theorem for
$J$-holomorphic curves, the above ingredients allow us to
construct a map on the punctured surface (see
Section~\ref{S:punc}). Finally, further arguments using the
extremal length, as well as a slightly more intricate one using in
addition the monotonicity lemma, allow us to show that limits can
also be obtained near the punctures in Section~\ref{S:fill}.

\subsection*{A word on notations}
We shall be considering various surfaces called $F_i$, $\Sigma_i$, $\Sigma$, $\bar{\Sigma}$ and $F$.
We clarify here what these mean(they will also be defined in appropriate places in the text).

The surfaces $F_i$ are the given surfaces of bounded mean curvature, taken with their pullback metrics. The inclusion map from the surface $F_j$ into $M$ will be cdenoted $i_j$
We shall frequently pass to subsequences without changing our notation.

We shall construct a surface $F$, which is topologically of the same type as $F_i$ with a pseudo-metric that
is as a limit of the pullback metrics. This will be the domain of the limiting map $i:F\to M$.

The surfaces $\Sigma_i$ are obtained from $F_i$ by deleting a finite set of points, with the metric on $\Sigma_i$
being the unique \emph{hyperbolic} metric that is conformally equivalent to the pullback metric on $F_i$. The natural inclusion maps of $\Sigma_j$ into $F_j$ and $M$ will be denoted $\iota_j$ and $\hat\iota_j$. We shall
construct a hyperbolic surface $\Sigma$ as a limit of the surfaces $\Sigma_i$. Finally, we shall compactify $\Sigma$
and make certain identifications at infinity to obtain a surface $\bar{\Sigma}$.

\tableofcontents
\setcounter{tocdepth}{1}

\section{Preliminaries}\label{S:prel}

We assume throughout that all manifolds (and surfaces) we consider are
orientable.  Let $M$ be a closed, smooth, Riemannian $n$-manifold
which we fix throughout. For a fixed real number $H_0\geq 0$, we
consider embedded surfaces $F\subset M$ with mean curvature bounded above in
absolute value by $H_0$. In case $H_0=0$, these are just minimal
surfaces. We further restrict to surfaces with area bounded above by
$A_0$ and with a fixed genus $g$. The proofs of Lemma \ref{up}
and Lemma \ref{geod} are given in Appendix A.

\subsection{Upper bounds on curvature}
We begin by observing that there is an upper bound on the sectional
curvature of the surfaces. This is used for a version of the Schwarz
lemma.

\begin{lemma}\label{up}
There is a constant $ K_0 $ so that the sectional curvature of each
surface $ F_{i} $ is bounded above by $ K_0 $.
\end{lemma}

\subsection{Lower bound on conjugate radius}
Let $(F,g)$ be a Riemannian manifold. The {\it conjugate radius}
at $p \in F$ is the largest $R$ such that $exp_p$ is an immersion
on $B(0,R) \subset T_p(M)$.

\begin{lemma}\label{geod}
Let $(F,g)$ be a complete Riemannian 2-manifold with sectional
curvature bounded above by $K_0$. Then the conjugate radius at any
$p \in F$ is at least $R = \frac {\pi}{3 \sqrt {K_0}}$. Moreover,
if we write
$$exp^\ast(g)=dr^2+f^2(r,\theta)d\theta^2$$
for polar coordinates $(r, \theta)$ on $T_p F$ and $r < R$, then
$f(r,\theta)$ increasing as a function of \ $r$ and
$f(r,\theta)>r/2$ for all $\theta$.
\end{lemma}

\subsection{Lifting discs under the exponential map}

We need to use the geodesic coordinates of Lemma~\ref{geod} for a topological disc $D\subset F$ in a surface
with an upper bound on the sectional curvature. However, the injectivity radius may be less than the diameter
of $D$. We shall see, however, that we can lift discs with small diameter and small boundary under the exponential
map. Let $R=R(K_0)$ be the constant from Lemma~\ref{geod}.

\begin{lemma}\label{explift}
Let $\iota:B\to (F,g)$ be an immersion of a disc into a complete
Riemannian $2$-manifold $(F,g)$ with sectional curvature bounded
above by $K_0$. Suppose  that for the pullback metric $i^*g$, the
length of  $\gamma=\del B$ and the distance of a point in $B$
to $\gamma$ are both bounded above by $\epsilon<R/10$ where $R =
\frac {\pi}{3 \sqrt {K_0}}$. Then for $x=\iota(y)$ in the image of
$B$, there is a lift $\tilde \iota$ of $\iota$ to the tangent space $T_x
F$ so that $\iota=exp_x\circ\tilde \iota$. Furthermore, the lift can be
chosen so that $\tilde \iota(y)$ is the origin.
\end{lemma}

We remark that $\iota$(at least restricted to the interior of the disc $B$) is often the inclusion
map on a subset of $F$. hence we identify $p$ with $x$.

We prove this in Section~\ref{S:lift}

\subsection{Conformal moduli of annuli}

We recall some basic results regarding the conformal moduli of
annuli that will be used extensively. An {\it annulus} is a
$2$-manifold homeomorphic to the product of a circle and an
interval. We consider annuli with a given conformal class of
Riemannian metrics (i.e., a conformal structure). Recall that this
is equivalent to specifying a complex structure.

A \emph{right circular annulus} $A(H,W)$ is the Riemannian product
of a circle of circumference $W$ and an open interval of (possibly
infinite) length $H$. The following is the uniformisation theorem
for annuli.

\begin{theorem} We have the following.
\begin{enumerate}
\item Any annulus $A$ with a conformal structure is conformally
  equivalent to a right circular annulus.
\item $A(H_1,W_1)$ is conformally equivalent to $A(H_2,W_2)$ if and
  only if $H_1/W_1=H_2/W_2$.
\end{enumerate}
\end{theorem}

From the above result, it is immediate that the following definition
gives a well-defined number in $(0,\infty]$.

\begin{definition}
The modulus $Mod(A)$ of an annulus $A$ with a conformal structure is
$Mod(A)=H/W$ where $A$ is conformally equivalent to the right circular
annulus $A(H,W)$.
\end{definition}

By definition the modulus is a conformal invariant. Further, two
annuli are conformally equivalent if and only if they have the same
modulus.

Observe that the area of a right circular annulus $A=A(H,W)$ is $HW$,
so the modulus of $A(H,W)$ can also be expressed as
$Mod(A)=Area(A)/W^2$, i.e., $W^2=Area(A)/Mod(A)$. The following
fundamental (though elementary) result of Ahlfors allows one to get an
upper bound on the appropriate width for an annulus.

\begin{theorem}[Ahlfors]\label{extrm}
Let $A$ be an annulus with a conformal structure. Then there is a
simple closed curve $\gamma\subset A$ separating the two boundary
components of $A$ whose length $l(\gamma)$ satisfies
$$l(\gamma)^2\leq Area(A)/Mod(A)$$ Furthermore, given an
identification of the annulus with $S^1\times J$ for an interval $J$,
we can find a curve $\gamma$ as above of the form $S^1\times \{p\}$.
\end{theorem}

It is easy to see that an annulus obtained from a disc by puncturing a
point has infinite modulus. Further, if we take an open disc $D(r)$ of
a fixed radius $r$ in Euclidean or hyperbolic space and $D(\rho)$ is
the concentric disc of radius $\rho<r$, then the modulus of the
annulus $A=D(r)-\overline{D(\rho)}$ goes to infinity as $\rho\to 0$.

\subsection{Real-analytic metrics and the cut-locus}\label{real}
Let $(F,g)$ be a Riemannian manifold and $p \in F$. Let

$$ U_p:= \{ v \in T_pM \ \vert \ exp_p(tv) \ {\rm is \ a \ minimal \
geodesic \ on} \ [0,1]\}.$$

Note that the boundary $\partial U_p$ is the cut-locus in $T_pM$.

In case $(F,g)$ is real-analytic we have the following proposition
which follows from results of S. B. Myers.

\begin{theorem}(S. B. Myers ~\cite{my1}, ~\cite{my2})\label{cut}
Let $(F,g)$ be a closed real-analytic Riemannian 2-manifold and $p
\in F$. Then

(i) $\partial U_p$ is a piecewise-smooth 1-manifold homeomorphic
to $S^1$.

(ii) Suppose that the sectional curvature of $(F,g)$ is bounded
above by $K_0$. Let $R= \frac {\pi}{3 \sqrt {K_0}}$ be the lower
bound on conjugate radius given by Lemma \ref{geod}. If $\de < R$,
then the geodesic sphere $\partial B(p, \de)$ is a disjoint union
of piecewise-smooth circles.
\end{theorem}

\section{A Monotonicity Lemma and $\epsilon$-nets}\label{S:monot}
As in the case of the monotonicity lemma for minimal surfaces, the
proof of our monotonicity lemma is based on an isoperimetric
inequality. The relevant isoperimetric inequality is due to
Hoffman and Spruck. This will also be used later on in the paper.

\begin{theorem}{\rm (Hoffman-Spruck ~\cite{hs})}\label{hs}
Let $\si$ be a compact surface with boundary in a Riemannian
n-manifold $(M,g)$. There is a constant $v_0=v_0(M,g)$ such that
either \ $vol (\si) \ge v_0$ or
\begin{equation}\label{gg}
vol ( \si)^{\frac {1}{2}} \ \le \ \be \ \Bigl ( vol (\partial \si
) + \int_{\si} \vert H \vert dV_\si \Bigr ),
\end{equation}
where $\be$ is an absolute constant.
\end{theorem}

\begin{proof}
This is a corollary of Theorem 2.2 of ~\cite{hs}. In the notation
of that paper take $\alpha = \frac {1}{2}$ and $b=1$, for
instance. Let $i_0$ denote the injectivity radius of $M$. Then
either (\ref{gg}) holds or
$$vol (\si) \ \ge v_0 = \frac {\omega_2}{2} \ min \bigl \{ 1 , \frac
{i_0^2}{ \pi^2} \bigr \} .$$
\end{proof}

The main result of this section is a monotonicity lemma, giving a
lower bound on the area of small balls.

\begin{theorem}\label{monot}

Let $(M,g)$ a Riemannian n-manifold and $\Sigma$ a compact surface
in $M$ with mean curvature $\vert H \vert \le H_0$. There exist
$c=c(M,g,H_0)$ and $\delta= \delta (M,g,H_0)$ such that the volume
of any ball of radius $\ep \le \de $ in $\si$ with the induced
metric satisfies
$$vol ( B(p, \ep)) \ \ge \  c \epsilon^2. $$
\end{theorem}

\begin{proof} By making an arbitrarily small $C^2$-perturbation we can assume
that $g$ is real-analytic. By Lemma ~\ref{cut} this implies that
the boundary of a ball of radius $\ep$ less than $R$ is
piecewise-smooth. We apply Theorem \ref{hs} to the manifold with
piecewise-smooth boundary $\partial B(p,\ep)$. Even (\ref{gg}) is
stated for submanifolds with smooth boundaries, it is clearly true
even if the boundaries are piecewise-smooth, as can be seen by
exhausting such manifolds by submanifolds with smooth boundaries.
We will apply it to the metric balls $B(p,r)$, $0<r \le \ep$.

If $vol(B(p,r)) \ge v_0$ for some $r < \ep$, then  $vol(B(p,\ep) > vol(B(p,r)) \ge v_0$. 
Hence $vol(B(p,\ep)) \ge \ep ^2$ if $\ep < \sqrt {v_0}$. 

So we can suppose that the
isoperimetric inequality (\ref{gg}) holds for every $r \le \ep$.

It follows from the co-area formula that $vol(B(p,r))  =  \
\int_0^r vol (\partial B(p,t)) dt$. Hence $vol(B(p,r))$ is
differentiable a.e. as a function of $r$.

By Theorem \ref{hs}, we then have

\begin{align}\notag
\frac {d }{dr} vol(B(p,r)) \ &=  \ vol (\partial B(p,r)) \notag \\
                            &\ge \ \be^{-1} vol ( B(p,r) )^{\frac
                            {1}{2}}-\int_{B(p,r)} \vert H \vert
                            dV_\si  \notag \\
                             &\ge \ \be^{-1} vol ( B(p,r) )^{\frac
                            {1}{2}}- H_0 \ vol ( B(p,r)) \ \ \ a.e. \notag \\
\end{align}

We next see that we can assume $\be^{-1} vol ( B(p,r)
)^{\frac{1}{2}}>2 H_0 \ vol ( B(p,r))$ for $0\leq r\leq \epsilon$.
If not, then we get $vol(B(p,\ep))\geq vol(B(p,r)\geq (\frac{1}{2
\be H_0})^2$, which is larger than $\epsilon^2$ for $\epsilon < 2 \be
H_0$.

As $\be ^{-1} vol ( B(p,r) )^{\frac{1}{2}}>2 H_0 \ vol ( B(p,r))$,
we get
$$\frac {d }{dr} vol(B(p,r))\geq \frac{1}{2 \be } vol ( B(p,r) )^{\frac{1}{2}}$$
By integrating, we get $vol(B(p,\epsilon)> \be ' \epsilon^2$ with $\be
'=\frac{1}{16 \be ^2}$.

Hence we can take
$$ \de = min \bigl \{ R, \ \sqrt {v_0}, \ \ 2 \be H_0 \bigr \},
\ \ \ \ c = min \{1, \ \frac{1}{16 \be ^2} \}.$$

\end{proof}

Gromov's proof of the compactness of $J$-holomorphic curves is
based on puncturing along an $\epsilon$-net. We shall choose an
appropriate constant $\epsilon=\epsilon(M,H_0, A_0, g_0)$, which
is the same for all the surfaces $S_i$. Assume that such a
constant has been chosen. For each surface $F_i$, we choose a
maximal subset $S_i\subset F_i$ so that the distance between every
pair of points in $S_i$ is at least $\epsilon$, i.e., $S_i$ is an
$\epsilon$-net.

\begin{lemma}
There is a constant $N$ such that for all $i$, $\vert S_i\vert\leq
N$.
\end{lemma}
\begin{proof}
Fix a surface $F_i$ in the sequence. By hypothesis, the open balls
$B(x,\epsilon)\subset F_i$, $x\in S_i$ are disjoint. By the
monotonicity lemma, there is a constant $a$ such that each of the
balls have area at least $a$. As the area of $F_i$ is bounded
above by $A_0$, the cardinality of $S_i$ is bounded above by
$A_0/a$.
\end{proof}

\section{Hyperbolic structures}\label{S:hyp}

Consider the sequence of surfaces $\hat F_j=F_j-S_j$. As the
cardinality of $S_j$ and the genus of $F_j$ are bounded above, by
passing to a subsequence we can, and do, assume that the surfaces
$\hat F_j$ are of a fixed topological type. Further, by ensuring that
the number of punctures is at least three, we can ensure that
$\chi(\hat F_i)<0$. By the uniformisation theorem, there is a unique
complete hyperbolic metric on $\hat F_j$ that is conformal to the
given Riemannian metric. We view this as a hyperbolic surface
$\Sigma_j$ which is identified with a subset of $F_j$. 

By the Bers-Mumford compactness theorem, on passing to a subsequence
the surfaces $\Sigma_j$ converge to a complete, finite volume, hyperbolic
surface $\Sigma$. More concretely, we have a sequence of numbers $\delta_j\to 0$, compact sets $\kappa_j\subset \Sigma_j$ and $(1+\delta_j)$-bi-Lipschitz diffeomorphisms $\psi_j:\kappa_j\to \Theta_j\subset\Sigma$ so that the sets $\Theta_j$ form an exhaustion of $\Sigma$. Furthermore, by passing to smaller sets, we can ensure that the sets $\Theta_j$ are complements of horocyclic neighbourhoods of the cusps of $\Sigma_j$ (with the intersection of the neighbourhoods of each cusp being empty).

We shall show that the maps $\iota_j\circ \psi_j^{-1}$ are
equicontinuous on compact sets (in a sense made precise below), so
that the Arzela-Ascoli theorem allows us to construct a limiting map
$\iota:\Sigma\to M$. We shall henceforth implicitly identify subsets of $\Sigma$ (contained in $\Theta_j$) with subsets of $\kappa_j$ using the maps $\psi_j$. Under these identifications, the maps $\iota_j$ can be regarded as maps on subsets of $\Sigma$.

Consider now a compact set $K\subset \Sigma$. The injectivity
radius on $K$ is bounded below by a constant $\alpha>0$. For $j$
large, as above we can identify $K$ with subsets $K_j\subset \Sigma_j$ and
the injectivity radius on these sets is also bounded below by
$\alpha$ as the map $\psi_j$ is $(1+\delta_j)$-bi-Lipschitz with $\delta_j$ small for $j$ large.

Our first step in proving equicontinuity is an upper bound on the
diameter in the pullback metric of small hyperbolic balls of a fixed
size.

\begin{lemma}\label{shortcurv}
There is a constant $r=r(\epsilon, \alpha, A_0, H_0, m, M)$ such
that for any point $x\in K_j$, $\iota_j(B(x,r))$ is contained in a
smooth (not in general metric) ball $B(\gamma)$ in $\widehat{F_j}\subset
F_j$ of diameter $3\epsilon$ whose boundary has length at most
$\epsilon$.
\end{lemma}
\begin{proof}
We shall choose $r<\alpha$ appropriately. For a point $x\in K_j$,
consider the annulus $A=B(x,\alpha)\setminus int(B(x,r))$. Choose
$r$ small enough that this annulus has modulus at least
$A_0/\epsilon^2$. Note that this depends only on $\alpha$, $A_0$ and
$\epsilon$.

Consider the annulus $\iota_j(A)\subset F_j$. This has area bounded
above by $A_0$. By Theorem~\ref{extrm}, there is a curve $\gamma$ in
$\iota_j(A)$, separating the boundary components of $A$, so that the
length of $\gamma$ is bounded above by $\epsilon$. The curve $\gamma$
is the boundary of a ball $B(\gamma)$ that contains
$\iota_j(B(x,r))$. We shall show that $B(\gamma)$ has diameter at most
$3\epsilon$.

As $B(\gamma)$ is a ball whose boundary is a connected set of diameter
at most $\epsilon$, it suffices to show that for each point $x\in
B(\gamma)$, the distance of $x$ from the boundary is at most
$\epsilon$. To see this, observe that $B(\gamma)\subset F_j-S_j$ by
construction. As $S_j$ is an $\epsilon$-net, the distance from $x\in
B(\gamma)$ to some point $y\in S_j$ is at most $\epsilon$. As the
metric on $F_j$ is obtained from a Riemannian metric and $F_j$ is
compact, there is a path $\beta$ from $x$ to $y$ of length at most
$\epsilon$. This path must intersect the boundary of $B(\gamma)$ at
some point $z$. It follows that $d(x,z)<\epsilon$.
\end{proof}

\section{A Schwarz lemma}\label{S:schwarz}

To deduce equicontinuity of the maps $\iota_j:\si_j \rt \hat F_j$
from the estimate on diameters in the induced metrics of balls of
radius $\rho$, we use an appropriate Schwarz lemma. \vspace{5mm}

\begin{theorem}\label{schwarz}
Let $(S,g)$ be a Riemannian 2-manifold and $p \in \si$ such that
the sectional curvature of $S$ $\le K_0$ and injectivity radius
$inj(p) \ge i_0$. Let $(\D,h)$ denote the unit disc with the
Poincar\'e metric of curvature $ -1$.

For any \ $r >0$, there exists $\eta=\eta (K_0, i_0)$ such that
for any conformal map \
$$f : B_h(0, r) \subset \D \rt B_g (p, \eta )\subset S,$$
with $f(0)=p$, we have
$$  \Vert df_0(v) \Vert_g \ \le \ r^{-2} \Vert v \Vert_h,$$
for all $ v \in T_0 \D$.
\end{theorem}
\begin{proof} The idea is to conformally deform $g$ on $B_g (p,
\eta )$ to $\tilde g=exp(2u)g$ so that $\tilde g$ has curvature
$\le -1$. The deformation will be done so that we have control
over the conformal factor. Then we can apply the Ahlfors-Schwarz
lemma in the new metric to  get the required estimate.

First choose $\eta < inj(p)$. Define $u:B_g (p, \de ) \rt \R^+$ by
$$u(x)= \la^2 \rho^2(x) \ \ {\rm and} \ \ \lambda = \sqrt
{\frac {K_0}{2}+1},$$ where $\rho(x)$ is the distance of $x$ from
$p$. Since $\de < inj(p)$, $u$ is smooth. Now the curvature
$\tilde K$ of $\tilde g = exp(2u)g$ is related to the curvature
$K$ of $g$ by
$$\tilde K = (K - \triangle u)exp(-2u),$$
where $\triangle$ is the negative Laplacian.

Since $K \le K_0$, \ if $ \rho(x) < min \{inj(p),\frac {\pi}{4
\sqrt {K_0} } \} $ , we can apply the usual comparison theorem
for the Laplacian of a distance function. Here the comparison
space $S_0$ is the sphere of curvature $K_0$. Let $p_0, x_0$
denote points in $S_0$ corresponding to $p,x$ and let $\rho_0$
denote the distance from $p_0$. We have $\rho_0(x_0)=\rho(x)$. The
comparison theorem gives
$$ (\triangle \ \rho) (x) \ge (\triangle_{0} \ \rho_0) (x_0) =
cot(\sqrt {K_0} \rho(x))= \frac {cos(\sqrt {K_0}
\rho(x))}{sin(\sqrt {K_0} \rho (x) )}.$$ In particular,
$$ (\triangle \ \rho) (x) \ge 0,$$
if $ \rho(x) < min \{inj(p),\frac {\pi}{4 \sqrt {K_0}} \} $. For
such $x$, we have

\begin{align}\notag
(\triangle u)(x) &= \la^2 (\triangle \rho^2)(x) \notag \\
            &= 2 \la^2 \rho(x) (\triangle \rho)(x) + 2 \la^2 \Vert d \rho \Vert^2 (x) \notag \\
            & \ge 2 \la^2. \notag \\
\end{align} \notag
 We then have
$$\tilde K (x) \le (K_0 - 2 \la^2)exp(-2 \rho(x)^2) = -2 \ exp(-2 \rho(x)^2)\le
-1,$$ for $\rho(x) < \sqrt {\frac {log (2)}{2}}$. Hence we can
take
$$\eta \ = \frac {1}{2} \ min \ \Bigl \{   i_0, \ \frac {\pi}{4 \sqrt
{K_0}}, \ \frac {log (2)}{2} \Bigr \}.$$
 Now we modify the
metric on $B_h(0, r)$. Note that if $\phi(z) = r^2 \frac {(1 -
\vert z \vert^2)^2}{ (r^2 - \vert z \vert^2 )^2}$, then $(B_h(0,
r), \phi^2h)$ is a complete Riemannian manifold of constant
curvature -1.

Let us recall the \vspace{2mm}

 {\it Ahlfors-Schwarz Lemma}: Let
$(F_1,g_1)$ be a complete Riemannian 2-manifold with curvature
$\ge -c$ and $(F_2,g_2)$ a Riemannian 2-manifold with curvature
$\le -d$, where $c,d>0$. If $f:(F_1,g_1) \rt (F_2,g_2)$ is a
conformal map, then
$$  \Vert df_x(v) \Vert_{g_2} \ \le \  \sqrt {\frac {c}{d}} \Vert v \Vert_{g_1},$$
for all $x \in F_1$, $v \in T_xF_1$. \vspace{2mm}

We apply the Ahlfors-Schwarz Lemma to $f : (B_h(0, r), \ \phi^2h)
\rt (B_g (p, \eta ), \ exp(2u)g)$ at $x=0$. Noting that \ $c=d=1$,
$\phi(0)= r^{-2}$ \ and \ $exp(2u(p))=1$, \ we get
$$  \Vert df_0(v) \Vert_g \ \le \ r^{-2} \Vert v \Vert_h.$$

\end{proof}

\section{Limits of punctured surfaces}\label{S:punc}

The Schwarz Lemma of Theorem~\ref{schwarz} requires a lower bound
on the injectivity of the pullback metric, which we cannot
control. However the following consequence of Lemma~\ref{explift}
allows us to obtain uniform Lipschitz bounds.

Note that the surface $\Sigma$ can be identified with the
complement of a collection of annuli in $\Sigma_j$ for each $j$.
Furthermore, if $K\subset \Sigma$ is a compact set, then for $j$
large the corresponding set $K_j\subset \Sigma_j$ is close to an
isometry.

Let $K\subset \Sigma$ be a compact set. Then there is a constant
$\alpha>0$ so that for $j$ large, the injectivity radius of the
set $K_j\subset \Sigma_j$ corresponding to $K$ is bounded below by
$\alpha$. Let $\iota_j:K_j\to F_j$ be the inclusion map.

\begin{lemma}\label{Lipsc}
There is a constant $\kappa$, independent of $j$, so that the map
$\iota_j$ is $\kappa$-Lipschitz for $j$ sufficiently large.
\end{lemma}
\begin{proof}
By Lemma~\ref{shortcurv}, there is a constant $r>0$ such that for
any $p \in K_j$,  $\iota_j(B(p,r))$ is contained in a topological
ball $B(\gamma)$ whose diameter is at most $3\epsilon$ and whose
boundary has length at most $\epsilon$ in the pullback metric.
Here $B(p,r)$ is the ball in the hyperbolic metric on $\Sigma_j$,

By Lemma~\ref{explift}, there is a lift $\tilde \iota_j$ of
$\iota_j$ to $T_p F_j$ under the exponential map $exp_p$. Choose
an isometry $\phi: B_h(0,r) \rt B(p,r)$, with $\phi (0) = p$,
where $B_h(0,r)$ denotes a ball in the Poincar\'e disc as in
Theorem \ref{schwarz}.

We can now apply the Schwarz Lemma of Theorem~\ref{schwarz} to

$$f=\tilde \iota_j \circ \phi: B_h(0,r) \rt B(0,R) \subset T_p F, $$
where $B(0,R)$ is endowed with the metric $exp_p^\ast g$.

Since $\iota_j = exp_p \circ \tilde \iota $ and $exp_p : (B(0,R),
exp_p^\ast g) \rt (F_j,g)$ is an isometry, we have

$$\Vert d \iota_j \Vert_{\phi(x)} = \Vert d \tilde
\iota_j \Vert_{\phi(x)}
 = \Vert df  \Vert_x.$$

Hence we obtain an upper bound on $\Vert d \iota_j \Vert$ at $p$
for the map $\iota_j:K_j\to \hat F_j\incl M$. Note that this upper
bound does not depend on $j$, but depends only on $M$ and
$\alpha$. This gives the uniform Lipchitz bound.

\end{proof}

\begin{lemma}\label{limpunc}
On passing to a subsequence, the maps $\hat\iota_j=i_j\circ\iota_j:\Sigma_j\to \hat
F_j\to M$ converge uniformly on compact sets to a map
$\iota:\Sigma\to M$.
\end{lemma}
\begin{proof}
Let $K\subset\Sigma$ be a compact set. We apply Lemma~\ref{Lipsc} to the restrictions of the maps $\iota_j$, regarded as maps on $K$, to obtain a uniform Lipshitz bound.
Thus, by the Arzela-Ascoli theorem, there is a subsequence of surfaces
so that the maps $\iota_j$ converge to a map $\iota:K\to M$. Now
consider an exhaustion of $\Sigma$ by compact sets $K^{(i)}$. By the
above, we can find a subsequence of surfaces to obtain a limit on
$K^{(1)}$. On passing to a further subsequence, we obtain a limit on
$K^{(2)}$. Iterating this process and using a diagonal subsequence as
usual, we obtain a limiting map on $\Sigma$.
\end{proof}

\section{Filling punctures}\label{S:fill}

We have constructed a limiting map from the punctured surface $\Sigma$
of finite type to the manifold $M$. We show now that, on passing to a
further sequence, we can construct a limiting map on a closed surface
$\bar{\Sigma}$ obtained from $\Sigma$. This is the surface by filling in
the punctures and making certain identifications of the filled in points.

Observe that there are two kinds of punctures (cusps). The first kind are those that correspond to the limits of punctures in $\Sigma_j=\hat{F_j}$ corresponding to points
of the $\epsilon$-net $S_j$. We denote the set of such punctures as $S(\Sigma)$. We can, and shall, identify these with points on the surfaces $F_j$.

The second kind are pairs of punctures formed in passing to the compactification of Moduli space by a sequence of curves $\alpha_j\subset \Sigma_j\subset F_j$ whose length in the hyperbolic metric  on $\Sigma_j$ goes to zero. We denote the set of such pairs of cusps by $\Lambda(\Sigma)$. Each point $p\in\Lambda(\Sigma)$ corresponds to a pair of ends $p^\pm$ of $\Sigma$.

We consider the Freudenthal (end-point) compactification of $\Sigma$ (where one point is added for each puncture) and identify points corresponding to pairs of punctures $p^\pm$, $p\in\Lambda(\Sigma)$. We denote the resulting surface by $\bar\Sigma$.

Thus,
$$\bar{\Sigma}=\Sigma\coprod \S(\Sigma)\coprod \Lambda(\Sigma)$$

We shall extend the inclusion map to $\bar{\Sigma}$. As there are only finitely many punctures, it suffices to show that we can extend the map to the point in $\bar{\Sigma}$ corresponding
to each puncture or pair of punctures.

We first consider a point $z\in S(\Sigma)$ corresponding to a limit of points of $z_j\in S_j$. As $M$ is compact, by passing to a subsequence we can ensure that $\iota_j(z_j)$ converges to a point, which we take to be the image of
$z$ in the limiting map. It remains to show that this extension is
continuous.

Suppose $D(z)$ is a closed disc in $\bar\Sigma$ containing the point
$z$ in its interior and no other points of $S(\Sigma)$ and $\Lambda(\Sigma)$. Then for $j$ sufficiently large, the disc $D(z)$
can be identified with discs $D_j=D_j(z)\subset F_j=\Sigma_j$. Continuity is immediate
from the following lemma.
\begin{lemma}\label{cont}
Given $\delta>0$, there is a disc $D(z)$ as above such that for $j$
sufficiently large, the diameter of $\iota_j(D_j)$ is at most
$\delta$.
\end{lemma}
We first give a brief sketch of the proof. Using an appropriate extremal
length argument, we enclose $D_j$ in a disc whose boundary has small
perimeter. As before, we lift this disc under the exponential
map. We then use polar co-ordinates, for which we have obtained the
appropriate estimates in~\ref{geod}. These allow us to deduce a bound
on the diameter of the disc from the bound on the perimeter. We now turn to the details.

\begin{proof}[Proof of Lemma~\ref{cont}]
As the modulus of a disc punctured at a
point is infinite, we can choose $D(z)$ so that there is an annulus
$A(z)$ enclosing the puncture corresponding to $z$ so that the
modulus of the annulus is at least $M$, where $M$ is any specified
number. The annulus $A(z)$ corresponds to an annulus $A$ in $F_j$
which, for $j$ sufficiently large, has modulus greater than $M$. Fix
such a $j$.

Choosing $M$ sufficiently large and using an extremal length argument
as in Lemma~\ref{shortcurv} , we can find a curve $\gamma$ in $A$
enclosing $z_j$ with length $L(\gamma)$ less than $\eta$ in the
pullback metric, with $\eta>0$ to be specified. Let $B(\gamma)$ be the
disc bounded by $\gamma$ enclosing $z_j$. We shall show that if $\eta$
is small enough (depending only on $\delta$ and $M$), then the disc
$D_j$ has diameter less than $\delta$. This implies that the diameter
of the image $\iota_j(F_j)$ is less than $\delta$.

We choose $\eta<min(\delta/2,\epsilon/2)$. As the length of the
boundary $\gamma$ of $B(\gamma)$ is bounded by $\delta/2$, it suffices
to show that the distance from a point $x_0$ of $B(\gamma)$ to
$\gamma$ is at most $\delta/3$. Suppose this is not the case, find a
point $x_0$ with distance from the boundary greater than $\delta/3$.

Consider the exponential map from the tangent space at $x_0$. We have
seen that this is an immersion on the ball of radius
$10\epsilon$. Further, as before the diameter of the set $D_j$ is
bounded by $2\epsilon$. We choose a lift as in Lemma~\ref{explift} so
that the image of $x_0$ is the origin.

We now recall Lemma \ref{geod}.

We have assumed  that the distance between $x_0$ and $\gamma$ is
greater than $\delta/3$. Hence the lift of $\gamma$ (which we
continue to denote by $\gamma$) encloses the ball of radius
$\delta$ around the origin.

Consider the radial projection $p:\gamma\to\alpha$ of $\gamma$ on to
the boundary $\alpha$ of the ball of radius $\delta/3$. As
$f(r,\theta)$ is an increasing function of $r$, this is distance
decreasing. Thus, if $ds_\gamma$ and $ds_\alpha$ denote the oriented
arc lengths of the respective curves, $p^{*}(ds_\alpha)=\psi\cdot
ds_\gamma$, with $\psi\leq 1$.

As $\gamma$ encloses the origin, the projection has degree one (after
possibly reversing the orientation of $\gamma$). Thus, we have
$$l(\alpha)=\int_\alpha ds_\alpha=\int_\gamma
p^{*}(ds_\alpha)=\int_\gamma \psi\cdot ds_\gamma\leq \int_\gamma
ds_\gamma=l(\gamma) $$

Hence $l(\alpha)\leq l(\gamma)<\eta$. Now by Lemma~\ref{geod}, it
follows that $l(\alpha)>\pi\delta/3$. Hence, as
$\eta<\delta/2<\pi\delta/3$ we get a contradiction.
\end{proof}

We now turn to the case of a point $p\in\Lambda(\Sigma)$, which corresponds to a pair of ends $p^\pm$. As
before, we can find disc neighbourhoods of these, $D^\pm(p)$,
enclosed by annuli $A^\pm(p)$  of modulus at least $\mu$ (see
figure~\ref{cusp}), with $\mu$ to be specified. We denote the
boundaries of the disc by $\xi^\pm=\del D^\pm$.

Consider now a surface $\Sigma_j$ with $j$ large. The curves $\xi^\pm_j$
in $F_j$ corresponding to $\xi^\pm$
enclose an annulus $B_j$ of large modulus (for $j$ large) and hence
are separated by a curve $\alpha_j=\alpha_j(p)$ of length $l_j$ (in the pullback metric on $F_j$) so that as $j\to
\infty$, $l_j\to 0$. Hence we can pass to a subsequence so that the
images $\iota_j(\alpha_j)$ converge to a point $y$ in $M$. We extend the
map so that the point $p$ maps to $y$. We show that
this extension is continuous.

\begin{lemma}\label{anncont}
Given $\delta>0$, there are discs $D^\pm=D^\pm(p)$ as above so that for $j$ sufficiently large the diameter of $B_j=B_j(p)$ is at most $\delta$.
\end{lemma}

Note that in this case, the upper bounds on the curvature and the
length of the boundary do not suffice, as there are flat annuli with
arbitrarily short boundaries with arbitrarily large diameters. We
shall use an indirect argument, which we sketch below.

Consider an annulus as above whose boundary curves have length
less than a small real number $\delta$. As the annulus is
contained in the complement of an $\epsilon$-net, we can bound the
distance between the two boundary components. This together with
the bound on the boundary components and the upper bound on
curvature gives an upper bound on the area of the annulus. By
ensuring $\epsilon$ (hence $\delta$) is sufficiently small, we can
thus ensure that we are not in the \emph{bubbling} case of the
isoperimetric inequality (i.e., we ensure that volume is less than
$v_0=v_0(g)$ of Theorem~\ref{hs}). Hence, as the upper bound
$\delta$ on the perimeter tends to zero, so does the area. Using
the monotonicity lemma, we can deduce that the diameter also tends
to $0$. We now turn to the details.

\begin{proof}[Proof of Lemma~\ref{anncont}]
As the annuli $A^\pm_j$ have modulus greater
than $\mu$, by choosing $\mu$ sufficiently large we can ensure that
there are curves $\gamma^\pm_j$ separating the boundary
components of the annuli with lengths at most $\epsilon$. We let $C_j=C_j(p)$
be the annulus enclosed by the curves $\gamma^\pm_j$. We shall show
that the diameter of $C_j$ is small, from which a bound on the
diameter of the image of $B_j$ follows.


Observe that as the annulus $C_j$ is contained in the complement of an
$\epsilon$-net in $F_j$, as before the boundary components are at most a
distance $2\epsilon$ apart. Namely, as before each point is a distance
less than $\epsilon$ from at least one of the boundary components. By
connectedness, some point must have distance less than $\epsilon$
from both the boundary components.

Choose an embedded arc $\beta$ of length at most $2\epsilon$ joining
the boundary components. We obtain a disc $\Delta_j$ from the annulus $C_j$
by splitting along the curve $\beta$. Observe that all the above
results continue to hold when the inclusion map of the disc is
replaced by the quotient map from the disc to $A_j$. Choosing
$\eta<\epsilon/2$, the disc $\Delta_j$ has boundary of length at most
$6\epsilon$.

As before, we pick a point $x_0$ in $B$ and find a lift of $\Delta_j$ with
respect to the exponential map based at $x_0$ so that the lift of
$x_0$ is the origin. Recall that there is a constant $v_0=v_0(B(g))>0$
associated with the isoperimetric inequality (Theorem~\ref{hs}) which
depends only on $M$.

\begin{lemma}
If $\epsilon>0$ is chosen sufficiently small, depending only on $M$,
then the area of $\Delta_j$ with the pullback metric is less than $v_0$.
\end{lemma}
\begin{proof}
As the the origin is in the interior of $\Delta_j$ and the distance
of any point in $B(\gamma)$ to the boundary is at most $\epsilon$, the
disc $\Delta_j$ is contained in the ball of radius $\epsilon$ around the
origin.

Using polar co-ordinates as before, the area form on $\Delta_j$ is
$\omega=f(r,\theta)dr\wedge d\theta$. Let
$F(r,\theta)=\int_0^rf(\rho,\theta)d\rho$ and
$\alpha=F(r,\theta)d\theta$. Then the area form $\omega=d\alpha$,
hence by Stokes theorem
$$Area=\int_{\Delta_j}\omega=\int_{\del \Delta_j} \alpha$$.

Let $\zeta=\del \Delta_j$. In polar co-ordinates, we can express
$\zeta=(\zeta_r,\zeta_\theta)$. In terms of these co-ordinates,

$$Area=\int_{\del \Delta_j} \alpha=\int F(r,\theta)\frac{d\zeta_\theta}{dt} dt$$

Note that as $f(r,\theta)$ is increasing as a function of $R$,
$F(r,\theta)\leq \epsilon f(r,\theta)$ for $r<\epsilon$. Further, by
Pythagoras theorem the oriented arc length $ds$ satisfies $ds>
\frac{d\zeta_\theta}{dt}$. Thus, it follows that

$$Area= \int F(r,\theta)\frac{d\zeta_\theta}{dt} dt<\epsilon l(\zeta)$$

As $l(\zeta)<5\epsilon$, the lemma follows.
\end{proof}

Assume $\epsilon>0$ has been chosen to satisfy the previous
lemma. Then, by the isoperimetric inequality, if $\eta$ is
sufficiently small, the volume of $C_j$ is less than $V(\eta)$,
where $V(\eta)$ is a function of $\eta$ such that $V(\eta)\to 0$
as $\eta\to 0$.

We show that if $\eta>0$ is sufficiently small, then the diameter of
the annulus $C_j$ is less than $\delta$. As before, it suffices to
show that the distance of each point $x_0\in C_j$ to the boundary of
the annulus is at most $\delta/4$.

Suppose the distance of $x_0$ to $\del C_j$ is greater than
$\delta/4$. It follows by the monotonicity lemma that the area of
$C_j$ is at least $V_\delta$, with $V_\delta$ depending only on
$\delta$ and $M$. Choose $\eta$ such that $V(\eta)<V(\delta)$, to
get a contradiction. Thus, we get a bound on the diameter showing
continuity as required. This completes the proof of Lemma~\ref{anncont}.

\end{proof}

\section{Proof of Theorem~\ref{mai}}\label{S:proof}

We can now complete the proof of Theorem~\ref{mai}. We have already constructed a limiting map on $\bar\Sigma$. We shall express $\bar{\Sigma}$ as the quotient of a surface $F$ and construct appropriate diffeomorphisms $\varphi:F_j\to F$ so that the maps $i_j\circ \varphi^{-1}$ converge.

Assume that we have chosen a subsequence so that we have a corresponding limit map $\iota$ on $\bar\Sigma$. Let $F$ be a surface of genus $m$. Identify $S(\Sigma)$ with a subset $S(F)$ of $F$. Then for a collection of disjoint curves $\alpha(p)\subset F-S(F)$, $p\in\Lambda(\Sigma)$, we have a homeomorphism
of $\Sigma$ with $F-\cup_{p\in\Lambda(\Sigma)}\alpha(p)-S(F)$. This extends continuously to a homeomorphism of the quotient of $F$ with each curve $\alpha(p)$ identified to a point, which is mapped to $p\in\bar\Sigma$. Choose and fix
a corresponding quotient map $q:F\to \bar\Sigma$. Let $i:F\to M$ be
the composition $i=\iota\circ q$. Observe that $\Sigma$ can be identified with a subset of $F$ so that $q$ the identity map on $\Sigma$ and $i=\iota$ on $\Sigma$.

We construct next diffeomorphisms $\varphi_j:F_j\to F$. These will be extensions of the diffeomorphisms $\psi_j:\kappa_j\to\Theta_j\subset\Sigma$ of Section~\ref{S:hyp} using the identification of $\Sigma_j$ with a subset of $F_j$.

We recall that the set $\Theta_j$ is the complement of a union of horocyclic neighbourhoods of cusps. Thus, there are punctured discs $\Delta(j;z)\subset\Sigma$, $z\in S(\Sigma)$ and $\Delta^\pm(j;p)\subset\Sigma$, $p\in\Lambda(\Sigma)$ so that $\Theta_j$ is the complement of the interiors of the sets $D(j;z)$ and $D^\pm(j;p)$. Without loss of generality we assume that $k\geq j$ implies that $\Delta(k;z)\subset \Delta(j;z)$ and $\Delta^\pm(k;p)\subset \Delta^\pm(j;p)$ for all $z$ and $p$.

We shall make use of the constructions of Lemmas~\ref{cont} and~\ref{anncont}. As the sets $\Theta_j$ form an exhaustion of $\Sigma$, for each fixed $z\in S(\Sigma)$ or $p\in\Lambda(\Sigma)$, the intersection of the corresponding punctured discs $\Delta(j;p)$ and $\Delta^\pm(j;p)$ is empty. It follows as in the proofs of Lemmas~\ref{cont} and~\ref{anncont} that if $D_k(z)$ and $B_k(p)$ denote the corresponding discs and annuli in $F_k$ for $k$ sufficiently large, then the diameters of these sets are bounded above by $\delta_j$ so that $\delta_j\to 0$ as $k\to\infty$.

We take $\varphi_j$ to be an extension of $\psi_j$ so that the complement of the set $\kappa_j\subset F_j$ is mapped to the complement of the set $\Theta_j\subset\Sigma\subset F$. Any compact set $K\subset\Sigma$ is contained in $\Theta_j$ for $j$ sufficiently large. Hence,
$i_j\circ\varphi^{-1}=i_j\circ\psi_j^{-1}$ converges to  $i=\iota:\Sigma\to M$. Finally, for the points of $S(F)$ and points on $\alpha_i$, by the continuity of the extension of $\iota:\Sigma\to M$ to $\bar\Sigma$, we see that $i_j\circ\varphi^{-1}:F\to M$ converges to $i=\iota\circ q:F\to M$.

\section{A limiting metric}

We continue to use the notation of the previous section. Using the
diffeomorphisms $\varphi_j$, we can identify the surfaces $F_j$
with $F$. Hence the pullback metrics on the surface $F_j$ give
Riemannian metrics on $F$ which have associated distance functions
$d_j$. We shall construct a limiting pseudo-metric $d$ on the
surface $F$. Recall that this is a symmetric function $d:F\times
F\to [0,\infty)$ that satisfies the triangle inequality but not in
general the positivity condition for metrics.

\begin{theorem}
On passing to a subsequence, the metrics $d_j$ converge uniformly on $F$ to a pseudo-metric $d$.
\end{theorem}
\begin{proof}
Let $g_0$ denote the fixed hyperbolic metric on $\Sigma$ and $d_0$ the
distance function of this metric. Our first step is to show that the distance functions $d_j$ converge on $\Sigma$.
\begin{lemma}
The family of function $d_j:\Sigma\times\Sigma\to\R$ is equicontinuous on
every compact subset $K$ of $\Sigma \times \Sigma$, where we consider
the product metric obtained from $d_0$ on $\si \times \si$.
\end{lemma}
\begin{proof}
Showing equicontinuity on the product is equivalent to showing that given $\eta>0$ there is a $\delta>0$ so that for pairs $(p_i,q_i)$, $i=1,2$, such that $d_0(p_1,p_2)<\delta$ and $d_0(q_1,q_2)<\delta$, we have $\vert d_j(p_1,q_1)-d_j(p_2,q_2)\vert<\eta$. Note that by the triangle inequality,
$$\vert d_j(p_1,q_1)-d_j(p_2,q_2)\vert<d_j(p_1,p_2)+d_j(q_1,q_2)$$

Hence it suffices to show that we can choose $\delta$ so that $d_0(p_1,p_2)<\delta$ implies $d_j(p_1,p_2)<\eta$. But this follows from  Lemma~\ref{Lipsc}, where we have shown the uniform Lipshitz property of the functions $d_j$ on compact subsets of $\Sigma$.
\end{proof}

We shall apply the above to the compact sets
$\Theta_j\in \Sigma=F$ of the previous section, which exhaust
$\Sigma$.  By the above Lemma the distance
functions $d_k$ are uniformly Lipshitz on $\Theta_j$ for each $j$. Hence,
we can iteratively pass to subsequences and use a diagonal sequence to
ensure that the metrics $d_k$ have a limit $d$ with convergence
uniform on each set $\Theta_j\times \Theta_j$. More precisely, we can ensure
that if $p,q\in \Theta_j$ and $k\geq j$, then $\vert d_k(p,q)-d(p,q)\vert<\delta_j'$ with $\delta_j'\to 0$ as $j\to\infty$.

As $d$ is the limit of metrics, it is a pseudo-metric on
$\Sigma$. Let $\bar{\Delta}(j;z)$ denote the closure of $\Delta(j;z)$ in $F$ and
$\bar{\Delta}(j;p)$ denote the closure of  $\Delta^+(j;p)\cup \Delta^-(j;p)$ in $F$. Observe that these sets are identified with sets $D_j(z)$ and $B_j(p)$ under the map $\varphi_j$ and hence have diameter at most $\delta_j$ in the metric $d_k$ for $k$ sufficiently large (as in the previous section). Hence the diameter of the sets $\bar{\Delta}(j;z)$ and $\bar{\Delta}(j;p)$ in the pseudo-metric $d$ is at most $\delta_j$.

We shall extend $d$ to $F$ by continuity. Consider first the case
where $p\in \Sigma$ and $q\in \alpha(p)$ for some $p\in\Lambda(\Sigma)$. Consider a sequence $q_i$ in
$\Sigma$ converging to $q$. Without loss of generality we can assume
that $q_j \in \bar{\Delta}(j;p)$ for some $p$. We claim that the sequence $d(p,q_j)$ is
Cauchy. For $k\geq j$, as $q_k\in \bar{\Delta}(k;p)\subset \bar{\Delta}(j;p)$ and the
diameter of $\bar{\Delta}(j;p)$ in the pseudo-metric $d$ is at most $\delta_j$, $d(q_j,q_k)\leq \delta_j$. Hence, by
the triangle inequality
$$\vert d(p,q_j)-d(p,q_k)\vert\leq  d(q_j,q_k)\leq \delta_j$$

It follows that the sequence $d(p,q_j)$ is Cauchy and hence
converges to a limit which we define to be $d(p,q)$. Observe that
if $q_j'$ is another sequence in $\bar{\Delta}(k;p)$, then as above $\vert
d(p,q_j)-d(p,q'_j)\vert<\delta_j$. Hence the limit is well-defined.

We can similarly define $d(p,q)$ if $q\in S(\Sigma)$. In
case neither $p$ nor $q$ are in $\Sigma$, we use sequences $p_j'$ and
$q_j'$ in $\Sigma$ converging to these points. As above we get Cauchy
sequences with limit independent of the choice of sequence.

We finally show that the convergence is uniform on all of $F$. Suppose
now that $p,q\in F$ are arbitrary. We shall find a uniform upper bound
for the quantity $\vert d_j(p,q)-d(p,q)\vert$. Suppose $p,q\in \Theta_j$,
then by as above $\vert d_j(p,q)-d(p,q)\vert\leq \delta_j'$. Otherwise,
one of $p$ and $q$ is in the interior of some set of the form $\bar{\Delta}(j;p)$
or $\bar{\Delta}(j;z)$. We consider the case when $p\in \Theta_j$ and $q\in \bar{\Delta}(j;p)$
as the other cases are similar.

Let $q'$ be a point in the boundary $\del \bar{\Delta}(j;j)=\bar{\Delta}(j;p)\cap
\Theta_j$. Then as above, $\vert d_j(p,q')-d(p,q')\vert<\delta_j'$. Further, as
the set $\bar{\Delta}(j;p)$ has diameter at most $\delta_j$ in the (pseudo)metrics
$d_j$ and $d$, $d(q,q')\leq \delta_j$ and $d_j(q,q')\leq \delta_j$

By the triangle inequality, it follows that
$$\vert d_j(p,q)-d(p,q)\vert\leq 2\delta_j+\delta'_j$$

This shows that we have uniform convergence of the metrics $d_k$ to
$d$ on all of $F$
\end{proof}

We shall use the notation
$$\vert d_j-d\vert := sup_{p,q\in F} \vert d_j(p,q)-d(p,q)\vert$$

\section{Fractal dimension of the limit}\label{Sec:dim}


We have constructed a limiting metric $d$ on the surface $F$. In this
section, we show that this metric has fractal dimension two and has finite, non-zero area in an appropriate sense. This gives a proof of Theorem~\ref{dim}

Now we come to the main lemma of this section.

\begin{lemma}\label{alp}
Let $(F,g)$ be a compact Riemannian 2-manifold with sectional
curvature $K \le K_0$.  For any $p \in F$ and $0<\de< \ \frac
{1}{\sqrt {3K_0}}$, if
$$ vol(B(p, \de)) > C \de^2,$$
then there is a $\de ' < \de$ with
$$ \int_{B(p,\de ')} Kdv \le 2 \pi - \frac {C}{2}.$$
\end{lemma}

\begin{proof}

 \vspace{2mm}
By making an arbitrarily small $C^2$ perturbation of $g$ we can
assume that it is real-analytic. Recall that $\frac {\pi}{3 \sqrt
{K_0}}$ is the uniform lower bound on the conjugate radius given
by Corollary \ref{geod}. Hence
$$ \tilde g : =exp^\ast(g)$$ is a Riemannian metric on the
(Euclidean) ball $B(0,\de)$ in $T_pM$. Also,
$$ exp_p: \ Int \ (U \cap B(0, \de)) \rt \ Int \ B(p,\de)$$
is a diffeomorphism, where $U=U_p$ as in Section \ref{real}. Note
that $U$ is star-shaped with respect to $0 \in T_pM$.

If we let $P= U \cap B(0,\de)$, then by Lemma \ref{cut} $\partial
P$ is a piecewise-smooth 1-manifold. In fact we can write

$$ \partial P = L_1 \cup ... \cup L_k \cup S_1 \cup...\cup
S_l,$$ where each $S_i$ is a smooth closed segment contained in
the circle $S(0,\de)$ and each $L_i$ is a smooth closed segment
contained in the cut-locus $\partial U$. We refer to the $L_i$ and
$S_i$ as {\it sides} of $P$ and the non-smooth points of $\partial
P$ as {\it vertices} of $P$. We can assume that a side intersects
another side in at most one point.

By changing $\de$ slightly, we can assume that vertices of $\del
P$ are either common points of a side $L_i$ and a side $S_j$ or
two sides $L_i$ and $L_j$.

By the definition of the cut locus, for each $L_i$ there is at
least one $L_j$, $j \neq i$, such that $exp_p(L_i)=exp_p(L_j)$.
For any such pair, $L_i \cap L_j = \phi$, since we know that
$exp_p$ is a local diffeomorphism on $B(0, \de)$.

\begin{lemma}\label{cutloc}
We have the following.
\begin{enumerate}
\item $ \sum_{i=1}^k \int_{L_i} \kappa =0$
\item The angle between any two consecutive sides of $P$ is
positive.
\end{enumerate}

\end{lemma}
\begin{proof}
As $\delta$ is less than the conjugacy radius, we have Gauss
normal coordinates on $P$ which we denote $r$ and $\theta$ as
usual, with the coordinates of a point $x$ denoted $r(x)$ and
$\theta(x)$.

Let $L_j$ be a segment as above parametrised by a function
$\alpha(s)$, with $\alpha(0)$ a point in the interior of $L_j$.
Then, as $L_j$ is in the boundary of $P$, there is a subsegment of
some segment $L_i$, parametrised by $\beta(s)$, so that
$\alpha(s)$ and $\beta(s)$ have the same image $c(s)$ under the
exponential map and $d(\alpha(s),p)=d(\beta(s),p)$, i.e.
$r(\alpha(s))=r(\beta(s))$, in a neighbourhood of $s=0$. The
images of the radial vectors joining $p$ to $\alpha(s)$ and
$\beta(s)$ form geodesics $\gamma_s$ and $\xi_s$ of the same
length joining $p$ to $c(s)$.

By differentiating $r(\alpha(s))=r(\beta(s))$ and considering images
in $F$, we see that the inner products of $\alpha'(0)$ and $\beta'(0)$
with the respective radial vectors are equal. As $\alpha(s)$ and
$\beta(s)$ have the same image on a neighbourhood of $0$, $\alpha'(0)$
and $\beta'(0)$ have the same norm. It follows that the angles made by
the vectors $\alpha'(0)$ and $\beta'(0)$ with the respective unit
radial vectors are the same. On passing to the image, we see that the
geodesics $\gamma_0$ and $\xi_0$ make the same angle with $c'(0)$ at
the point $c(0)$.

It follows that $\gamma_0$ and $\xi_0$ must approach $c(0)$ on
opposite sides -- otherwise they would have a common point and
direction and hence coincide. In particular, there are exactly two
points in $\del P$ that map to a smooth (even $C^1$) point on the
cut-locus, for if there were at least three points two would be on
the same side. It follows that the segments $L_j$ are identified
in pairs, with the interior of $P$ mapping to opposite sides of
the image.

We deduce that

$$ \sum_{i=1}^k \int_{L_i} \kappa =0,$$ where $\kappa$ denotes
geodesic curvature. This is because the terms that correspond to
$L_i$ and $L_j$ mapping to the same segment have equal magnitude
(as their image is equal) and opposite signs (as the interior of
$P$ maps to opposite sides of the image). 

We next see that the internal oriented angle $\theta_i$ between
any two consecutive closed segments in $U \cap B(0,\de)$ is
positive. Observe that as $P$ is a star convex region with
boundary piecewise smooth, we can parametrise $\del P$ by angle
using a function $r(\theta)$. This is smooth wherever $\del P$ is
smooth. The left and right derivatives $r'_\pm(\theta)$ exist at
all points. All internal oriented angles are positive if and only
if for every non-smooth point (i.e., vertex) $\theta$,
$r'_-(\theta)\geq r'_+(\theta)$.

Consider first the case when a vertex of $\del P$ between an edge of
the form $L_i$ and one of the form $S_j$. By construction, on $S_j$ we
have $r(\theta)=\delta$ and on $L_i$ we have $r(\theta)\leq
\delta$. It is immediate that $r'_-(\theta)\geq r'_+(\theta)$.

Next, consider a vertex $v$ between segments $L_{i_1}$ and
$L_{i_2}$, parametrised by $\alpha_1(s)$, $s\leq 0$ and
$\alpha_2(s)$, $s\geq 0$, respectively. By the above, there are
edges $L_{j_1}$ and $L_{j_2}$ that can be parametrised by curves
$\beta_1$ and $\beta_2$ with the images of $\alpha_i$ and
$\beta_i$ coinciding and $d(p,\alpha_i(s))=d(p,\beta_i(s))$.

We see that the curves $\alpha_i$ and $\beta_i$ can be extended so
that their domain of definition includes a neighbourhood of the
origin and with the images of $\alpha_i$ and $\beta_i$ coinciding
and $d(p,\alpha_i(s))=d(p,\beta_i(s))$. Without loss of
generality, we prove this for $\alpha=\alpha_1$ and
$\beta=\beta_1$. Namely, as $\delta$ is less than the conjugacy
radius, the exponential map gives diffeomorphisms from
neighbourhoods $U_\alpha$ and $U_\beta$ of $\alpha(0)$ and
$\beta(0)$ to a neighbourhood $V$ of their image $v$. The images
of the coordinate function $r$ under these diffeomorphisms gives
coordinate functions $r_\alpha$ and $r_\beta$. The condition
$d(p,\alpha_i(s))=d(p,\beta_i(s))$ is equivalent to
$r_\alpha=r_\beta$.

The gradients of the functions $r_\alpha$ and $r_\beta$ at $v$ are
unit vectors along the geodesic segments from $p$ to the vertex $v$
that are the images of the radial vectors to $\alpha(0)$ and
$\beta(0)$. As these geodesics do not coincide, the gradients do not
coincide and hence the gradient of $r_\alpha-r_\beta$ is non-zero. It
follows that the set $r_\alpha=r_\beta$ is a manifold near $v$. Taking
inverse images under the diffeomorphisms from $U_\alpha$ and $U_\beta$
to $V$ gives the required smooth extensions of $\alpha$ and $\beta$.

We now consider these extensions of $\alpha_1$ and $\alpha_2$. By
the definition of the cut-locus, it follows that for $s>0$,
$r(\alpha_2(s))\leq r(\alpha_1(s))$. As $\alpha_1(s)$ is smooth at
$0$, we deduce that $r'_-(\theta)\geq r'_+(\theta)$ if
$\theta=\alpha(0)$ is the given vertex. Thus, in this case too the
angles are positive. This completes the proof of the lemma.

\end{proof}
We now return to the proof of Lemma \ref{alp}.  Since $\de < \frac
{\pi}{3 \sqrt {K_0}} $, we can compare the Riemannian manifold
$(B(0, \de), \tilde g)$ with the round $2$-sphere of radius $\frac
{1}{\sqrt K_0}$ (the advantage of working with $B(0, \de)$ in
$T_pM$ rather than $B(p, \de)$ in $M$ is that the injectivity
radius of $B(0,\de)$ with the pull-back metric $\tilde g$ is
$\de$). If $\kappa$ denotes the mean-curvature function on $S(0,
\de)$, then

$$ \kappa \ \ \ge \ \ \sqrt K_0 \ \frac {\cos (\sqrt K_0 \de)} {\sin
(\sqrt K_0 \de)} \ \ \frac {1}{2 \de} \ \ \ \ {\rm on} \ \ \ \
S(0,\de).$$



Now we can apply the Gauss-Bonnet theorem to get

\begin{align}\notag
\int_{B(p,\de)}K \ =& \ \int_{U \cap B(0,\de)} \tilde K \\ \notag
                =& \ 2 \pi \chi (B(0,\de)) - \sum_{i=1}^k \int_{L_i} \kappa -
                \sum_{j=1}^l \int_{S_j} \kappa- \sum_{i=1}^k \theta_i \notag \\
                 \le & \ 2 \pi - \sum_{j=1}^l \int_{S_j} \kappa  \notag \\
                  \le & \ 2 \pi - \frac {l(\de)}{2 \de}, \label{tta}
\end{align}
where $\tilde K$ is the Gaussian curvature of $\tilde g$ and
$$l(\de) = \sum_{j=1}^l \ length (S_j).$$
Note that we have used the Lemma~\ref{cutloc} proved earlier in going from
line 2 to line 3 above. The area of $B(p,\de)$ is given by
\begin{align} \notag
vol(B(p,\de)) = vol (U \cap B(0,\de))= \int_0^\de l(s)ds, \notag
\end{align}
Hence if $vol(B(p,\de)) \ge C \de ^2$ for some $C$, then there exists $\de ' \le \de$ with
$$l(\de ') \ge C \de .$$
By (\ref{tta}), we would have for this $\de '$
$$ \int_{B(p, \de ')}K  \ \le \  2 \pi - \frac {l(\de ')}{2 \de '} \ \le \ 2 \pi - \frac {C}{2}.$$

This completes the proof of Lemma~\ref{alp}.

\end{proof}

\begin{corollary}\label{lower}
Let $(F,g)$ be a compact Riemannian 2-manifold with sectional curvature $K \le K_0$
and area $A$ satisfying $a_0 \le A \le A_0$.
Suppose that $S=\{x_1,...,x_l\}$ is a $\de$-net with
$\de< \frac {1}{\sqrt {2K_0}}$.

Then, for any $x_i \in S$, we have
$${\rm vol} \ B(x_i,\de) \le C_0 \de^2,$$
where
$$C_0= 8 \pi^2 (1- \chi(F)) + 4 \pi K_0A_0.$$

Hence it follows that the cardinality of $S$ is at least $\frac {a_0} {C_0} \de^{-2}$.
\end{corollary}
\begin{proof}

We will assume that at some point, say $x_1$,
$${\rm vol} \ B(x_1,\de) \ge C_0 \de^2$$
and get a contradiction.
The Gauss-Bonnet theorem applied to $F$ along with Lemma \ref{alp} gives
\begin{align} \notag
2 \pi \chi (F) & = \int_{B(x_1,\de ')}Kdv + \int_{F-B(x_1,\de ')}Kdv \\ \notag
               & \le 2\pi - \frac {C_0}{2} + K_0A_0. \\ \notag
\end{align}
This gives $$C_0< 4 \pi (1- \chi(F)) +  K_0A_0,$$ a contradiction.
\end{proof}

Let $F$ be a surface in $M$ with the given bounds on mean
curvature, genus and area bounded above by $A_0$. Then the
sectional curvature of $F$ is bounded above. We next see that
there is a lower bound $a_0$ on the area of $F$ depending only on
the geometry of $M$ and the given bounds on $F$. This allows us
to apply the above corollary uniformly.

\begin{lemma}
There is a constant $a_0$ depending only on the geometry of $M$
and the bound on the mean curvature of $F$ such that the area of
$F$ is at least $a_0$.
\end{lemma}
\begin{proof}
This follows from Theorem~\ref{monot} applied for an arbitrary value of $\epsilon$.
\end{proof}

We next show that there is a lower bound on the area of a ball of radius $\delta$, hence an upper bound on
the size of a $\delta$-net.

\begin{lemma}\label{upper}
 There is a constant $c>0$ such that for $\delta$ sufficiently small, the area of the ball of radius $\delta$
 in $F$ around a point $p\in F$ is at least  $c\delta^2$. As a consequence the size of a $\delta$-net is at most
 $A_0/c(\delta/2)^2$.
\end{lemma}
\begin{proof}
Again, by using Theorem ~\ref{monot}, we can deduce the bound on
the size of the $\delta$-net. Note that for a $\delta$-net $S$,
the balls of radius $\delta/2$ centered around the points of $S$
are disjoint. Hence their total area is at most the area of $F$,
which is in turn at most $A_0$. As the area of each of these balls
is at least $c(\delta/2)^2$, it follows that the cardinality of
$S$ is at most $A_0/c(\delta/2)^2$.
\end{proof}

We conclude that the size of a $\delta$-net for the metric $d$ grows as $\delta^{-2}$ as $\delta\to 0$.

\begin{theorem}\label{fracdim}
There are constants $0<b<B<\infty$ such that, for $\delta$
sufficiently small, the size of a $\delta$-net $S$ for the
pseudometric $d$ satisfies
$$b\delta^{-2}\leq\vert S\vert\leq B\delta^{-2}$$
\end{theorem}
\begin{proof}
Suppose $\delta>0$ is sufficiently small and $S$ is a $\delta$-net for $d$, i.e., a maximal set so that all pairwise distances are at least $\delta$. Let $j$ be such that $\vert d_j-d\vert<\delta/2$. Then for $p,q\in S$, $p\neq q$, we have $d_j(p,q)\geq \delta/2$. Hence $S$ is contained in a $\delta/2$-net $S'$ for the metric $d_j$. But Lemma~\ref{upper} gives an upper bound of the form $B\delta^{-2}$ for the cardinality of $S'$, and hence of $S$.

Next, let $T$ be a $3\delta$-net for the metric $d_j$. We claim that the cardinality of $T$ is at most that of $S$. First observe that as $S$ is a $\delta$-net, if $x\in F$ then for at least one $p=p(x)$ in $S$, $d(x,p)\leq \delta$. If $x\in S$, this is obvious, otherwise be considering $S\cup \{x\}$ we get a contradiction to the masimality. As  $\vert d_j-d\vert<\delta/2$, it follows that $d_j(x,p(x))<3\delta/2$.

For each point $q\in T$, choose and fix $p(q)$ as above. This gives a function $p:T\to S$.

\begin{lemma}
$p:T\to S$ is injective.
\end{lemma}
\begin{proof}
Suppose $p(q)=p(q')=p$. Then we have seen that $d_j(q,p)\leq 3\delta/2$ and $d_j(q',p)<3\delta/2$. By the triangle inequality, $d_j(q,q')<3\delta$, contradicting the hypothesis that $T$ is a $3\delta$-net for the metric $d_j$.
\end{proof}
It is immediate that the cardinalities of $S$ and $T$ satisfy $\vert T\vert\leq \vert S\vert$. But Corollary~\ref{lower} gives a lower bound of the form $b\delta^{-2}$ on $\vert T\vert$, hence on the cardinality of $S$.

\end{proof}

A coarse notion of area (and volume), and the corresponding notion of dimension, the so called \emph{fractal dimension}, can be defined in terms of $\delta$-nets.
Namely, let $(X,d)$ is a metric space. For $\delta>0$, let $n(\delta)$ be the minimum number of balls of radius $\delta$ that cover $X$. For $s.0$, define the $s$-dimensional volume by
$$V^s(X)=\limsup_{\delta\to 0} n(\delta)\delta^s$$

It is an immediate consequence of Theorem~\ref{fracdim} that the
$2$-dimensional volume, in the above sense, of $F$ with the
metric $d$ is a finite, positive number. Further for $s<2$ the
$s$-dimensional volume is zero and for $s>2$ it is infinite.
Thus, the limiting metric on the surface is a
metric of fractal dimension two and of finite, positive $2$-dimensional volume.

The fractal dimension is closely related to, but not equal to, the
Hausdorff dimension. In particular, it is a \emph{capacity} rather
than a measure - we have finite but not countable additivity. For
example, if $X$ is the set $\mathbb{Q}\cup [0,1]$ if rational numbers
in $[0,1]$, then $V^1(X)=1$ but the $1$-dimensional Hausdorff measure
of $X$ is zero.

It is easy to deduce from the above that the Hausdorff dimension of $F$ is at most $2$. However, we do not know whether the Hausdorff dimension must be two. We remark that the metric $d_j$ is not in general bilipshitz to the pseudometric $d$ as for pairs of distinct points $p$, $q$ in a circle $\alpha^i$ in $F$ as above (if there is at least one such circle), $d(p,q)=0$ but $d_j(p,q)\neq 0$.

\section{Appendix A: Bounds on curvature and conjugate radius}\label{S:bdcurv}

\begin{lem}[see Lemma~\ref{up}]
Let $F$ be an embedded surface in a Riemannian $n$-manifold
$(M,g)$ with mean curvature bounded above by $H_0$. There is a
constant $ K_0=K_0(M,g,H_0) $ so that the sectional curvature of $
F$ is bounded above by $ K_0 $.
\end{lem}

\begin{proof}
Let
$\tilde \nabla$ and $\nabla$ denote the Riemannian connections of
$M$ and $F$. Fix $p \in F$ and let $N_1,..,N_{n-2}$ be unit
normal vector fields defined in a neighbourhood of $p$. Then the
second fundamental form $B$ is given by
$$B(X,Y):= \widetilde \nabla_XY - \nabla_XY $$
and can be written as
$$B(X,Y) \ = \ \sum_{i=1}^{n-2} \langle B_i(X),Y)N_i \rangle \ \ \ \ \ X,Y \in T_pF$$
where the symmetric linear operators $B_i:T_pF \rt T_pF$ are given
by $B_i(X) = - (\widetilde \nabla _X N_i)^T$. The mean curvature
field is given by
$$H \ = \ \sum_{i=1}^{n-2}Tr(B_i)N_i .$$

If $k_a$ is the sectional curvature of $M$ along the tangent plane
$T_pF$ then, by the Gauss-Codazzi formula, the sectional curvature
of $F$ at $p$ is given by
$$k=k_a+ \sum_{i=1}^{n-2} Det(B_i).$$

Fix $i$ for now. Let $\kappa_1$ and $\kappa_2$ denote the
eigenvalues of $B_i$. Since
$$\vert H \vert^2 = \sum_{i=1}^{n-2}Tr(B_i)^2,$$
we have $ \vert \kappa_1 + \kappa_2 \vert \le
\vert H \vert \le H_0 $.

Hence
$$Det(B_i)=\kappa_1\kappa_2=\frac{(\kappa_1+\kappa_2)^2-(\kappa_1-\kappa_2)^2}{4}
\leq\frac{H_0^2}{4}.$$ Since $M$ is compact, there is an upper
bound $K_M$ on $k_a$. It follows that the sectional curvature of $F$  is
bounded above by $ K_0=K_M+(n-2)H_0^2/4$.

\end{proof}

\begin{lem}[see Lemma~\ref{geod}]
Let $(F,g)$ be a complete Riemannian 2-manifold with sectional
curvature bounded above by $K_0$. Then the conjugate radius at any
$p \in F$ is at least $R:= \frac {\pi}{3 \sqrt {K_0}}$. Moreover,
if we write
$$exp^\ast(g)=dr^2+f^2(r,\theta)d\theta^2$$
for polar coordinates $(r, \theta)$ on $T_pF$ and $r < R$, then
$f(r,\theta)$ increasing as a function of \ $r$ and
$f(r,\theta)>r/2$ for all $\theta$.
\end{lem}
\begin{proof}
Fix polar coordinates $(r, \theta)$ on $T_pF$. $exp$ will denote
$exp_p$. We know that
$$exp_\ast \Bigl ( \frac {\partial}{\partial r} \Bigr ) = \frac {\partial}{\partial
r}.$$ Let $$J_\theta (r) = exp_\ast \vert_{(r,\theta)} \Bigr (
\frac {\partial}{\partial \theta} \Bigr ), \hspace{1cm} f(r,
\theta)= \Vert J_\theta(r) \Vert .$$ Note that

$$exp^\ast(g)=dr^2+f^2(r,\theta)d\theta^2.$$
Fix $\theta >0$ and regard $f$ as a function $r$ alone. Let $T>0$ be the smallest value of $T$ such that $J_\theta(T)=0$. Then $f$ is
smooth on $[0,T)$. Assume that $r \in (0,T]$.
Since $J$ is a Jacobi field, we have

$$K(r, \theta){f(r,\theta)} \ = \ -  \frac { \partial^2
f (r, \theta)}{ \partial r^2 },$$ where $K$ denotes the Gaussian
curvature of $F$. Therefore
$$ f'' + K_0f \ge 0.$$
 The above inequality combined with
$$ f(0)=0, \ \ f'(0)=1 \ \ {\rm and} \ \ f \ge 0$$
implies that
\begin{equation}\label{eh}
 \sin (\sqrt {K_0} r)f '(r)- \sqrt {K_0} \cos
(\sqrt {K_0} r)f (r) \ge 0
\end{equation}
for
$$ r \in  I= \Bigl [0,\ \frac {\pi}{2 \sqrt K_0} \Bigr ].$$ Hence $$f'(r) \ge \sqrt
{K_0} \cot (\sqrt {K_0} r)f (r) \ge 0$$ on $I$. This implies that
$f>0$ on $I$, since $f$ is non-decreasing on $I$ and $f(r)=0$ if
and only if $f'=0$ on $(0,r)$, which would contradict $f'(0)=1$.

Integrating (\ref{eh}), we get
$$f(r) \ge \frac {f(t)}{\sin(\sqrt K_0 t)}{\sin(\sqrt K_0 r)}$$
for $0 < t< r < \frac {\pi}{2 \sqrt K_0}$.

Letting $t \rightarrow 0$,
$$f(r) \ge \frac {1}{\sqrt K_0} \sin(\sqrt K_0 r) > \frac {r}{2}$$
for $ 0<r < \frac {\pi}{3 \sqrt K_0}$.
\\
Hence we can take $R = \frac {\pi}{3 \sqrt K_0}$.

\end{proof}

\section{Appendix B: Lifting discs under the exponential map}\label{S:lift}

In this section, we prove Lemma~\ref{explift} which allows us to lift discs to the tangent space.

\begin{lem}[see Lemma~\ref{explift}]
Let $\iota:B\to (F,g)$ be an immersion of a disc into a complete
Riemannian $2$-manifold $(F,g)$ with sectional curvature bounded
above by $K_0$. Suppose  that for the pullback metric $i^*g$, the
length of  $\gamma=\del B$ and the distance of a point in $B$
to $\gamma$ are both bounded above by $\epsilon<R/10$ where $R =
\frac {\pi}{3 \sqrt {K_0}}$. Then for $x=\iota(y)$ in the image of
$B$, there is a lift $\tilde \iota$ of $\iota$ to the tangent space $T_x
F$ so that $\iota=exp_x\circ\tilde \iota$. Furthermore, the lift can be
chosen so that $\tilde \iota(y)$ is the origin.
\end{lem}

\begin{proof}
Recall (Lemma~\ref{geod}) that $exp$ is an immersion on a ball of
uniform radius $R$ in $T_pF$ for any $p\in F$. Hence $exp_x^\ast (g_j)$ is a metric
on this ball and the injectivity radius at $0$ is at least $R$. The tangent space $T_x F$ is a vector
space with an inner product with origin identified with $x$. We
seek lifts with respect to the exponential map $exp_x:T_x F\to
F$.


By the inverse function
theorem, Lemma~\ref{geod} yields the following.

\begin{lemma}\label{smalldisc}
There is a constant $\delta>0$ such that given any point $\xi=exp_x(z)$,
with $z\in B(0,9\epsilon)\subset T_x F$, there is a map
$\exp_z^{-1}:B(\xi,\delta)\to T_x F$ with $exp_x\circ exp_\xi^{-1}$ the
identity map and $exp_z^{-1}(\xi)=z$.
\end{lemma}

Thus, the exponential map is invertible on sets of diameter less than
$\delta$ containing a point in the image of $B(0,9\epsilon)$. Note
that $\delta$ is not universal. However, none of the constants in the
Schwarz lemma depend on $\delta$.  Observe that as the metric on the domain $B$ of $i$ is the pullback metric, a set of diameter at most $\delta$ in $B$ has image of diameter at most $\delta$.

We shall construct a lift on the disc $B$ by inductively
lifting sets $b_i$ of small diameter, as in the proof of the
homotopy lifting theorem in Algebraic Topology. However, we need
to ensure that at each stage the lift remains within
$B(0,9\epsilon)$ to continue the process. In our situation we can
indeed choose such sets $b_i$ using a geometric argument making
use of the fact that $\del B(\gamma)$ has length less than
$\epsilon$.

\begin{lemma}
There is a sequence of smooth balls $b_i$, each of which has diameter
at most $\epsilon$, so that if $B_j=\cup_{i=1}^j b_i$ and $B_0$ is a
single point $B_0=\{y_0\}$,
\begin{enumerate}
\item The set $B_j\cap b_{j+1}$ is connected and non-empty for $j\geq
0$.
\item The set $B_j\cap \del B$ is connected and non-empty for
  each j.
\item For each point $p\in B_j$, there is a path $\alpha$ contained in
  $B_j$ of length at most $2\epsilon$ joining $p$ to $B_j\cap \del
  B$.

\end{enumerate}
\end{lemma}
\begin{proof}
We shall first construct discs $b_i$ with diameter at most $\delta$
and then re-order them to satisfy the condition of the lemma.

Consider the function $f:B\to\R$ given by the distance to
the boundary. This is positive on the interior of $B$ and
vanishes on the boundary. After a small perturbation, we can
assume that this is Morse. Clearly the function $f$ has no local
minima in the interior of $B$.

Thus, $f$ has finitely many critical points of index one and of index
$2$. As there are no local minima, the descending manifold of each
critical points of index $1$ is a pair of arcs joining the critical
point to the boundary. These partition the disc $B$ into
closed subdiscs which we call \emph{basins}. Each of these basins $P$
is the closure of the descending submanifold (i.e., the basin of
repulsion of a critical point of index $2$ which we denote $O(P)$. We
regard the gradient lines from the local maximum in a subdisc as
radial lines (see figure~\ref{poly}).


Consider now one such basin $P$. This is a polygon with $2k$ sides for
some $k\geq 0$, with alternate sides contained in the boundary of
$B(\gamma)$ and alternate sides consisting of an index one critical
point and the descending submanifolds of these.

Consider a closed interval $J$ in $P\cap \del B$. We define the
cone $C(J)$ to be the closure of the set of gradient lines that end in
$J$. If $J$ is in the interior of $P\cap \del B$ and has
boundary points $a$ and $b$, then $C(J)$ is the region enclosed by $J$
together with the gradient lines joining $O(P)$ to $a$ and $b$ (see figure~\ref{poly}).

Suppose next that one endpoint of $J$ is a vertex $v$ of $P$ and the
other is an interior point $a$. The vertex $v$ is the limit of the
gradient line joining an index-one critical point $x$ to $V$. The cone
$C(J)$ is then the region enclosed by $J$, the gradient line from
$O(P)$ to $a$, and an arc consisting of the closure of the gradient
line from $O(P)$ to $x$ and the gradient line from $x$ to $v$.

In both these cases, we can identify the cone with a sector in the
circle, with gradient lines identified with radial lines. Using such
an identification, the cone is foliated by lines transversal to the
radial lines, namely those corresponding to lines of a fixed distance
from the vertex of the sector, which we call \emph{longitudinal arcs}
$\lambda_r$. We call the point identified with the centre of the
circle the \emph{centre} of the cone and the arc $J$ the
\emph{boundary arc}.

By bounding the length of $J$ from above, we can ensure that the
length of each arc $\lambda_r$ is less than $\delta/2$. We
subdivide the boundary $\del B$ into closed arcs $J_k$
such that each arc $J_k$ is contained in some basin $P_{i_k}$
with at most one endpoint a vertex, and with the lengths of the
arcs $J_k$ sufficiently small to ensure that the corresponding
longitudinal arcs in the cones have lengths at most $\delta/2$.
We get a partition of $B$ into corresponding subsets
$C_k=C(J_k)$.

We can further partition $C_k$ into regions between pairs of
longitudinal arcs. For each $C_k$ we choose a collection of
longitudinal arcs such that each of the regions between pairs of
longitudinal arcs has diameter at most $\delta$. We shall call these
regions \emph{squares} (even though the region containing the centre
is really a triangle).

This gives a partition of $B$ into discs $\delta$. After
re-ordering, these discs will be the regions $b_i$. Observe that the
regions in $C_k$ are naturally ordered starting with the region
containing the centre and ending with the region containing the
opposite arc. We shall use this as well as the opposite order. We
shall often specify whether the first square is the one containing the
centre or the boundary arc and consider the corresponding natural
order.

We begin by ordering the arcs $J_k$. Pick an arc $J_1$ with both
endpoints in the interior of an edge of a basin. Order the arcs
cyclically beginning with the edge $J_1$. Let $P_1$ be the basin
containing $J_1$ and let $y_0$ be a point in $J_1$.

We shall now order the squares $b_i$ (see figure~\ref{order} showing
$B_j$ at various stages). Consider the cone on the arc $J_1$ and let
$b_1$, $b_2$, \dots $b_{l_1}$ be the regions of $C_1$ in the natural
order so that $b_1$ contains $J_1$. By construction, for $i\leq l$,
$B_i\cap \del B=J_1$ and each point in $B_i$ is connected to
$J_1$ by a radial line of length less than $\epsilon$.


Next, let $k$ be such that $J_2$, \dots $J_k$ are contained in
$P$ and $J_{k+1}$ is not (this includes the case when there are
only $k$ arcs $J_i$). We let $b_{l_1+1}$ be the square in $J_2$
containing the centre of $C_2$ and let $b_{l_2+2}$, \dots,
$b_{l_2}$ be the other squares in $C_2$ in the natural order.
Observe that for $l_1<i\leq l_2$, each point in $B_i$ can be
connected to the centre by a radial line of length at most
$\epsilon$. The centre can in turn be connected to $J_1\subset \del
B\cap B_i$ by a radial line of length at most $\epsilon$.
It is easy to see that the other claims also hold for the sets
$b_i$ and $B_i$ constructed so far.

We now continue this process inductively, choosing $b_{l_2+1}$ to
be the square of $C_3$ containing the centre and then choosing
successive regions by the natural order. The same argument
verifies the claims for these cases. In this manner, we can order
all regions in the cones $J_1$, \dots $J_k$ to get $b_1$,\dots
$b_l$.

Next consider (if we have not exhausted $B$) the cone
$C_{k+1}$. We take the next square $b_{l+1}$ to be the square in
$C_{k+1}$ that contains $J_{k+1}$. As before, the regions
$b_{l+2}$, \dots, will be the successive regions in $C_{k+1}$ in
the natural order upto the region containing the centre. As for
the first basin, for the successive interval $J_{k+2}$,\dots
$J_{k'}$ in the same basin as $J_{k+1}$, we take regions in
successive intervals ordered starting with the region containing
the centre.

The above constructions repeated inductively give an ordering of the
regions $b_i$ satisfying all the claims.

\end{proof}

\begin{lemma}
Given a set $B_j$ as above and two points $p,q\in B_j$, there is a
path in $B_j$ of length at most $5\epsilon$ joining $p$ to $q$.
\end{lemma}
\begin{proof}
The points $p$ and $q$ can be joined to points $p'$ and $q'$,
respectively, in $B_j\cap \del B$ by paths of length at most
$2\epsilon$. As $B_j\cap \del B$ is connected and the length
of $\del B$ is at most $\epsilon$, $p'$ and $q'$ can be joined
by a path of length at most $\epsilon$.
\end{proof}

We now complete the proof of Lemma~\ref{explift}. We construct inductively lifts $\tilde \iota_j$ on the sets $B_j$. First
note that as $y_0$ and $x$ are in the set $B$ whose diameter
is at most $3\epsilon$, there is a point $z_0\in B(0,3\epsilon)\subset
T_x F$ with $exp_x(z_0)=\iota(y_0)$. We define the map $\tilde \iota$ on $B_0$ by $\tilde
\iota(y_0)=z_0$.

Next, we inductively construct a map $\tilde \iota_{j+1}$ on $B_{j+1}$
extending the given map on $B_j$. Doing this is equivalent to
extending the lift on $B_j$ to the set $b_{j+1}$. First observe that,
for a point $\xi\in B_j\cap b_{j+1}$, there is a path $\beta$ joining
$\xi$ to $y_0$ of length less than $5\epsilon$. By considering the path
$\tilde \iota_j\circ \beta$, it follows that $z=\tilde \iota_j(\xi)$ is
contained in the ball of radius $8\epsilon$ in $T_x F_j$. Hence by Lemma~\ref{smalldisc}
we can construct a map $\exp_z^{-1}:\iota(b_{j+1})\to B(0,10\epsilon)$
which is a local inverse for the exponential map and so that
$\exp_z^{-1}(\iota(\xi))=z$. We define $\tilde \iota_{j+1}$ on $b_{j+1}$ as
$exp_z^{-1}\circ \iota$. Note that this agrees with the previous
definition on $\xi$. By the inverse function theorem, for each $y_1\in
B_j\cap b_{j+1}$, the inverse image under the exponential map in
$B(0,10\epsilon)$ of $y_1$ is a discrete set. As $B_j\cap b_{j+1}$ is
connected and both $exp_z^{-1}\circ \iota$ and $\tilde \iota_j$ give lifts of
$\iota$ on $B_j\cap b_{j+1}$ of the exponential map that agree at the point $\xi$, it follows that
$exp_z^{-1}\circ \iota$ agrees with $\tilde \iota_j$ on $B_j\cap b_{j+1}$. It
follows that we have a well-defined extension $\tilde \iota_{j+1}$ of
$\tilde \iota_j$. Proceeding inductively we obtain a lift $\tilde \iota$ as
claimed.

We can ensure that $y$ lifts to the origin by picking a path $\alpha$,
from $y$ to $y_0$ of length at most $3\epsilon$. A simpler variation
of the above argument gives a lift of this path to a path
$\tilde\alpha$ beginning at the origin and ending at some point $z_0$.
We proceed as before with $\tilde \iota(y_0)=z_0$.

\end{proof}

\bibliographystyle{amsplain}

\end{document}